\documentclass[12pt, oneside]{amsart}

\usepackage[T1]{fontenc}
\usepackage[utf8]{inputenc}

\usepackage{amssymb}
\usepackage{amsthm}
\usepackage{amsmath}
\usepackage{mathrsfs}

\usepackage[a4paper,left=2cm,top=3cm,right=2cm,bottom=2cm]{geometry}

\usepackage{enumerate}

\usepackage{indentfirst}

\usepackage{hyperref}
\hypersetup{colorlinks=true,allcolors=black}  

\usepackage{tikz-cd}
\tikzcdset{arrow style=math font}

\newcommand\actionarrow[2]{#1 \begin{tikzcd}[ampersand replacement=\&, cramped, sep=small] \& \arrow[loop left] \phantom{X} \end{tikzcd} \hspace{-9pt} #2}

\DeclareMathOperator{\differential}{d}
\DeclareMathOperator{\diam}{diam}
\DeclareMathOperator{\codim}{codim}

\DeclareMathOperator{\symrank}{symrank}

\DeclareMathOperator{\vol}{vol}

\DeclareMathOperator{\ric}{Ric}

\DeclareMathOperator*{\spannn}{span}

\DeclareMathOperator{\reg}{reg}

\newcommand{\Ho}{H}

\newcommand{\blup}[2]{\mathrm{Bl}_{#1}\hspace{1pt}{#2}}
\newcommand{\Slash}{/\mkern-4mu/}

\mathchardef\ordinarycolon\mathcode`\:
\mathcode`\:=\string"8000
\begingroup \catcode`\:=\active
  \gdef:{\mathrel{\mathop\ordinarycolon}}
\endgroup

\theoremstyle{definition}
\newtheorem{example}{Example}[section]
\newtheorem{definition}[example]{Definition}

\theoremstyle{plain}
\newtheorem{prop}[example]{Proposition}
\newtheorem{thm}[example]{Theorem}
\newtheorem{lemma}[example]{Lemma}
\newtheorem{cor}[example]{Corollary}

\newtheorem{thmx}{Theorem}
\newtheorem{corx}[thmx]{Corollary} 

\newcommand{\dif}[0]{\differential\hspace{-1pt}}

\newcommand{\metric}{\ensuremath{\mathrm{g}}}

\newcommand{\fol}[0]{\mathcal{F}}
\newcommand{\vct}[0]{\mathfrak{X}}
\newcommand{\trs}[0]{\mathfrak{l}}

\newcommand{\RR}[0]{\mathbb{R}}

\DeclareMathOperator{\Ann}{Ann}
\newcommand{\Gr}[0]{\mathrm{Gr}\hspace{1pt}}

\hypersetup{
  pdftitle={Infinitesimally polar foliations via geometric resolutions},
  pdfauthor={Francisco C. Caramello Jr.; Laura Ribeiro dos Santos}
}

\title[Geometric resolutions]{Infinitesimally polar Killing foliations via geometric resolutions}

\author{Francisco C.~Caramello Jr.}
\address{Departamento de Matemática, Universidade Federal de Santa Catarina, R. Eng. Agr. Andrei Cristian Ferreira, 88040-900, Florianópolis - SC, Brazil}
\email{francisco.caramello@ufsc.br}
\author{Laura Ribeiro dos Santos}
\address{Instituto de Matemática e Estatística, Universidade de São Paulo, Rua do Matão, 1010, 05508-090, São Paulo - SP, Brazil}
\email{laura.ribeiro.santos@usp.br}

\date{}
\keywords{Singular Riemannian foliations, geometric resolution, infinitesimal polarity, Killing foliations, strong Molino conjecture, basic cohomology, transverse curvature}

\subjclass[2020]{Primary 53C12; Secondary 53C20, 55N91, 57R18, 57R30}

\begin{document}

\begin{abstract}
We establish the strong Molino conjecture for complete infinitesimally polar singular Riemannian foliations and identify their Molino sheaves with those of their geometric resolutions. The resolution preserves basic and equivariant basic cohomology and transverse curvature bounds. These comparisons yield orbifold models and approximations of leaf-closure spaces, Euler localization, closed-leaf existence and Euler positivity under positive curvature, and transverse sphere theorems. Maximal defect gives a closed-ball leaf-closure space. Further applications include fundamental-group growth bounds and closedness under nonpositive transverse Ricci curvature that is negative definite at one regular point.
\end{abstract}

\maketitle
\setcounter{tocdepth}{1}
\tableofcontents

\section{Introduction}

When $\fol$ is a complete regular Riemannian foliation, the closures of its leaves form a singular foliation $\overline{\fol}$. Blowing a minimal stratum up projectively reduces the depth of the foliation, and finitely many steps produce a regular Riemannian foliation $\blup{}{\fol}$. In \cite{molinodesing}, P.~Molino used this approach on $\overline{\fol}$, which later formed part of the development of the structural theory subsequently presented in \cite{molino}, that describes $\overline{\fol}$ via a locally constant sheaf $\mathscr{C}_\fol$ of transverse Killing fields. When $\mathscr{C}_\fol$ is globally constant, $\fol$ is a Killing foliation. The iterated blow-up construction was later extended by M.~Alexandrino to arbitrary singular Riemannian foliations \cite[Theorem~1.2]{alexdesing}. A central part of the construction is the choice of adapted metrics on the blow-up spaces, which involves gluing locally defined metrics, and hence has the downside of altering the geometry of the leaf space near the desingularized part.

In our previous work \cite{blupcaramello}, we studied projective blow-ups from the viewpoints of dynamics and basic cohomology. For a complete Riemannian foliation $\fol$, we compared the Molino sheaf of $\blup{}{\fol}$ with the inverse image of $\mathscr{C}_\fol$ \cite[Proposition~3.2 and Theorem~4.3]{blupcaramello}, obtaining in particular that a singular Killing foliation on a compact manifold blows up to a regular Killing foliation \cite[Corollaries~4.4--4.5]{blupcaramello}. Combining this comparison with the change of the ordinary Euler characteristic under real projective blow-up gave closedness criteria, for instance, $\overline{\fol}=\fol$ when the ambient Euler characteristic is nonzero and all singular strata have odd codimension \cite[Theorem~4.6]{blupcaramello}. The cohomological comparison in that work, however, highlights another downside of the desingularization via blow-ups: the basic version of the usual comparison, relating the basic cohomology of $\fol$ and $\blup{}{\fol}$, requires a strong hypothesis concerning the injectivity of the induced blow-down map on the level of basic forms \cite[Theorem D]{blupcaramello}.

A complementary development concerns desingularizations built from horizontal sections. In this vein, the Grassmannian constructions for polar geometry in \cite{boualem,tobenparallel} were applied by A.~Lytchak in order to build a similar desingularization technique for a large class of Riemannian foliations; namely, the so called infinitesimally polar foliations \cite[Sections~2.4 and~4.2]{lytchak}. This desingularization, known as \emph{Lytchak's canonical geometric resolution}, has far better topological and geometric properties. By replacing a point with the space of sections of its infinitesimal foliation, it produces a smooth submanifold $\Gr M$ of a Grassmannian bundle, with a regular Riemannian foliation $\Gr\fol$. The resolution map $\Gr M \to M$ is proper, is a foliated diffeomorphism over the regular locus, and induces a bijection between leaf spaces. Its stronger metric preservation comes with the precise restriction of infinitesimal polarity, identified in \cite[Theorem~1.4]{lytchakthorbergsson}: infinitesimal polarity at a point is equivalent to local closedness and the existence of a local Riemannian orbifold quotient; it is also equivalent to local boundedness of the regular transverse sectional curvatures. Alexandrino and Lytchak subsequently proved that pullback through the geometric resolution identifies the sheaves of basic differential forms \cite[Proposition~3.5]{alexlytchak}, supplying an isomorphism between the basic cohomologies of $\fol$ and $\Gr\fol$.

Alexandrino and M.~Radeschi established in \cite{alexradeschiclosure} the Molino conjecture: the closures of the leaves of every complete singular Riemannian foliation form another singular Riemannian foliation. The stronger assertion that every prescribed section of the Molino sheaf admits a smooth representative, known as the \emph{strong Molino conjecture}, is still open in general \cite[p.~215]{molino}\cite[Section~9.4]{surveyalexcaramello}\cite[Section~3.1]{caramellosingularequivariant}. It was established for orbit-like foliations in \cite[Theorem~3.4]{caramellosingularequivariant}. Here we establish it for the infinitesimally polar class in Theorem~\ref{theorem strong molino infinitesimally polar}, using Lytchak's geometric resolution. We also show that the Molino sheaves of $\fol$ and $\Gr\fol$ are isomorphic:

\begin{thmx}[Theorems~\ref{theorem strong molino infinitesimally polar} and~\ref{theorem molino resolution}]
Let $(M,\fol,\metric)$ be a complete infinitesimally polar singular Riemannian foliation, and let $\rho\colon \Gr M\to M$ be its geometric resolution. Then:
\begin{enumerate}[(i)]
\item The Molino sheaf $\mathscr C_{\fol}$ is smooth; in fact, every local section has a unique smooth horizontal foliated representative.
\item There is a canonical isomorphism of sheaves of Lie algebras $\rho^{-1}\mathscr C_{\fol}\cong\mathscr C_{\Gr \fol}$.
\end{enumerate}
\end{thmx}

It follows that every complete, infinitesimally polar singular Riemannian foliation on a simply connected manifold is Killing. The smoothness proof applies the descent of basic one-forms to each prescribed Molino section, a mechanism already present in Alexandrino and Lytchak's proof that the leaf closures of an infinitesimally polar foliation form a singular Riemannian foliation \cite[Theorem~3.6 and Remark~3.1]{alexlytchak}. Here we make the argument explicit for the entire Molino sheaf. A conormal description of the resolution then lifts smooth foliated representatives, and uniqueness of transverse Killing fields identifies their germs with the Molino fields upstairs.

\begin{corx}[Corollaries~\ref{corollary killing resolution} and~\ref{corollary killing criterion infinitesimally polar}]
If $\fol$ is an infinitesimally polar singular Killing foliation, then $\Gr \fol$ is a complete, regular Killing foliation, with $\mathfrak{a}_{\fol}\cong\mathfrak{a}_{\Gr\fol}$.
\end{corx}

The resolution provides a second comparison, involving cohomology and curvature. As already mentioned, pullback by $\rho$ preserves basic differential forms and thus induces an isomorphism $\rho^*\colon \Ho^*(\fol)\to \Ho^*(\Gr \fol)$ in basic cohomology (Proposition \ref{proposition basic forms resolution}). We combine this fact with the preservation of the structural action to obtain that $\rho$ also identifies their equivariant basic cohomologies, as graded $S(\mathfrak{a}^*)$-algebras (Proposition \ref{proposition equivariant cohomology resolution}). Concerning transverse curvature, we show that, on $\rho^{-1}(M_{\reg})$, the transverse curvature and Ricci tensors of $\Gr\fol$ are the pullbacks of those of $\fol$ by $\rho$ (Theorem \ref{theorem transverse curvature resolution}). In particular, uniform bounds on those curvatures persist at the exceptional points (Corollary \ref{corollary transverse curvature bounds resolution}). The applications connect these comparisons to the deformation method for regular Killing foliations presented by A.~Haefliger and E.~Salem in \cite[Theorems~3.4 and~3.7]{haefligersalem}. In \cite{positivecaramellotoben}, F.~Caramello and D.~Töben studied the deformation of the underlying geometry and the invariance of the basic Euler characteristic under these transformations. Their later work also studied the preservation of equivariant basic cohomology, with respect to the transverse infinitesimal action of the structural algebra $\mathfrak{a}\cong\mathscr{C}_\fol(M)$ \cite[Theorem~5.3 and Remark~5.4]{equivariantcaramellotoben}. Caramello and F.~Neubauer also used these deformations to obtain transverse sphere theorems in the context of Killing foliations \cite[Theorems~4.1,~5.2 and~5.3]{spherecaramelloneubauer}. We say that $d:=\dim\mathfrak{a}$ is the \textit{defect} of $\fol$, the maximal absolute difference between the dimensions of a regular leaf and its closure. Since $\mathscr{C}_\fol$ describes $\overline{\fol}$, in the sense that the leaf closures are the $\mathfrak{a}$-orbits that is, $\mathfrak{a}\fol=\overline{\fol}$, we see that the fixed point set $M^{\mathfrak{a}}$ consists precisely of the union of all closed leaves of $\fol$.

Here we obtain several applications for infinitesimally polar
Killing foliations by transferring those regular results through
geometric resolutions. We first collect the consequences for
orbifold models, cohomology, and metric approximation:

\begin{thmx}[Theorems~\ref{theorem closed orbifold models applications}
and~\ref{theorem GH approximation applications}, and
Corollary~\ref{corollary Euler localization applications}]
Let $\fol$ be an infinitesimally polar singular Killing
foliation of codimension $q$ on a compact connected manifold $M$. Then:
\begin{enumerate}[(i)]
\item There is a compact $q$-dimensional Riemannian orbifold
$\mathcal O$, endowed with an effective isometric
$\mathbb{T}^d$-action, such that there is a homeomorphism
\[\frac{M}{\overline{\fol}}\cong\frac{|\mathcal O|}{\mathbb{T}^d},\]
and $\chi(\fol)=\chi(|\mathcal O|)$. Moreover, $\Ho_{\mathfrak{a}}(\fol)\cong\Ho_{\mathbb{T}^d}(\mathcal O)$, as graded $\mathbb{R}$-algebras, and if $\actionarrow{\mathfrak{a}}{\fol}$ is equivariantly formal, then $\actionarrow{\mathbb{T}^d}{\mathcal O}$ is equivariantly formal as well, and $b_j(\fol)=b_j(|\mathcal O|)$, for every $j$.

\item The basic Euler characteristic localizes to the locus of closed leaves:
\[\chi(\fol)=\chi(M^{\mathfrak{a}}/\fol).\]
In particular, $\chi(\fol)\neq0$ implies that $\fol$ has a closed
leaf.

\item There is a sequence of compact $q$-dimensional Riemannian
orbifolds $(\mathcal O_i,\mathrm{h}_i)$ such that
\[(|\mathcal O_i|, d_i)\xrightarrow{\mathrm{GH}}M/\overline{\fol},\]
where $d_i$ is the distance induced by $\mathrm{h}_i$, and $\chi(|\mathcal O_i|)=\chi(\fol)$. Given fixed strict constant bounds on transverse sectional or Ricci curvature that hold on $M_{\reg}$ with a uniform margin,
the sequence can be chosen so that every $(\mathcal O_i,\mathrm{h}_i)$ satisfies those bounds. If $\actionarrow{\mathfrak{a}}{\fol}$ is equivariantly formal, one can moreover require $b_j(|\mathcal O_i|)=b_j(\fol)$ for every $i,j$.
\end{enumerate}
\end{thmx}

The localization formula extends the scope of
\cite[Theorem~D]{caramellosingularequivariant}, while the metric
approximation enhances \cite[Theorem~C]{blupcaramello} in the
infinitesimally polar case. Among the geometric consequences,
we obtain the following.

\begin{thmx}[Theorem~\ref{theorem positive curvature applications} and Corollary~\ref{corollary maximal even defect applications}]
Let $\fol$ be an infinitesimally polar singular Killing foliation of codimension $q\geq2$ on a compact manifold $M$, with defect $d$, and suppose that $\sec_{\fol}\geq\kappa>0$ on $M_{\reg}$.
\begin{enumerate}[(i)]
\item The defect satisfies
\[d\leq\left\lfloor\frac{q+1}{2}\right\rfloor.\]
If equality holds, then $M/\overline{\fol}$ is homeomorphic to the closed ball $\mathbb{D}^{q-d}$.
\item If $q$ is even, then $\fol$ has a closed leaf. If, moreover, $d\geq q/4-1$, then $\chi(\fol)>0$. If $q=2d$, then there are finitely many closed leaves and their number is $\chi(\fol)$.
\item If $\actionarrow{\mathfrak{a}}{\fol}$ is equivariantly formal, then
\[\sum_{j=0}^q b_j(\fol)\leq B,\]
where $B$ is a constant depending only on $q$.
\end{enumerate}
\end{thmx}

We also obtain some conclusions under constraints on the transverse Ricci curvature.

\begin{thmx}[Corollary~\ref{corollary positive Ricci applications} and Theorem~\ref{theorem negative curvature applications}]
Let $\fol$ be an infinitesimally polar singular Killing foliation of codimension $q\geq1$ on a compact manifold $M$.
\begin{enumerate}[(i)]
\item If $\ric_{\fol}^{\intercal}\geq\lambda$ on $M_{\reg}$ for some $\lambda>0$, then $\Ho^1(\fol)=0$. If $q\in\{2,4\}$, then also $\chi(\fol)>0$ and $\fol$ has a closed leaf.
\item If $\ric_{\fol}^{\intercal}\leq0$ on $M_{\reg}$ and $\ric_{\fol}^{\intercal}$ is negative definite at some point of $M_{\reg}$, then $\fol$ is closed.
\end{enumerate}
\end{thmx}

Further sectional curvature applications concern the transverse Grove--Shiohama and Berger--Klingenberg sphere theorems.

\begin{thmx}[Theorems~\ref{theorem diameter sphere applications} and~\ref{theorem quarter pinched applications}]
Let $\fol$ be an infinitesimally polar singular Riemannian foliation of codimension $q\geq2$.
\begin{enumerate}[(i)]
\item If $M$ is compact, $\fol$ is Killing, $\sec_{\fol_{\reg}}\geq\kappa>0$, and $\diam(M/\overline{\fol})>\pi/(2\sqrt\kappa)$, then
\[\frac{M}{\overline{\fol}}\cong\frac{\mathbb{S}^q/\Gamma}{\mathbb{T}^d},\]
for a finite subgroup $\Gamma<\operatorname{O}(q+1)$ and a continuous action $\actionarrow{\mathbb{T}^d}{\mathbb{S}^q/\Gamma}$, where $d=\dim\mathfrak{a}$.
\item If $\fol$ is complete, $q\geq3$, and $0<\alpha\leq\sec_{\fol_{\reg}}\leq\beta<4\alpha$, then there are a Riemannian metric $\mathrm{h}$ on $\mathbb{S}^q$ and a countable subgroup $\Gamma<\operatorname{Iso}(\mathbb{S}^q,\mathrm{h})$ such that there is a homeomorphism and an isometry
\[\frac{M}{\fol}\cong \frac{\mathbb{S}^q}{\Gamma}\quad \text{and} \quad \frac{M}{\overline{\fol}}\cong\frac{(\mathbb{S}^q,\mathrm{h})}{\overline\Gamma},\]
respectively. If moreover $\fol$ is closed, then $\Gamma$ is finite and $M/\fol\cong \mathbb{S}^q/\Gamma$ is an orbifold diffeomorphism.
\end{enumerate}
\end{thmx}

Finally, we obtain growth bounds for fundamental groups of leaf spaces and a basic Betti number estimate under negative pinching.

\begin{thmx}[Theorems~\ref{theorem nonnegative Ricci growth applications} and~\ref{theorem negative sectional growth applications}]
Let $\fol$ be an infinitesimally polar singular Riemannian foliation of codimension $q\geq1$ on a compact manifold $M$.
\begin{enumerate}[(i)]
\item If $\ric_{\fol}^{\intercal}\geq0$ on $M_{\reg}$ and $\fol$ is closed, then both $\pi_1^{\mathrm{orb}}(M\Slash\fol)$ and $\pi_1(M/\fol)$ have polynomial growth of degree at most $q$.
\item If $q\geq2$, $\fol$ is Killing, and $-b^2\leq\sec_{\fol}\leq-a^2<0$ on $M_{\reg}$ for constants $0<a\leq b$, then $\fol$ is closed, $\pi_1^{\mathrm{orb}}(M\Slash\fol)$ grows exponentially, and
\[\sum_{j=1}^q b_j(\fol)\leq N b^q\vol(M\Slash\fol),\]
where $N$ is a constant depending only on $q$.
\end{enumerate}
\end{thmx}

The paper is organized as follows. Section~\ref{section background} fixes the definitions and recalls the geometric resolution. Section~\ref{section lifting} construct transverse lifting. Section~\ref{section molino resolution} establishes the strong Molino conjecture for the infinitesimally polar class and identifies the Molino sheaves. Section~\ref{section cohomology curvature} treats basic cohomology and transverse curvature, and Section~\ref{section applications} collects the applications.

\section{Background}
\label{section background}

Throughout the paper, all manifolds are smooth and without boundary, and all vector fields and maps are smooth unless otherwise stated. We assume that $M$ is connected. When completeness is required for $(M,\fol,\metric)$, it refers to the chosen adapted Riemannian metric $\metric$ on $M$. If $M$ is compact, this completeness is automatic. For an open subset $U\subset M$, the leaves of $\fol|_U$ are the connected components of the intersections of leaves of $\fol$ with $U$.

\begin{definition}[Singular Riemannian foliation]
Let $(M,\metric)$ be a Riemannian manifold and let $\fol$ be a partition of $M$ into connected immersed submanifolds, called \emph{leaves}. We say that $\fol$ is a \emph{singular Riemannian foliation} if
\begin{enumerate}
\item for every $x\in M$ and $v\in T_x\fol(x)$, there exists a smooth vector field $V$, tangent to all leaves, such that $V(x)=v$;
\item every geodesic perpendicular to a leaf remains perpendicular to all leaves it meets.
\end{enumerate}
Here $\fol(x)$ denotes the leaf containing $x$. The second condition is also expressed by saying that $\metric$ is \emph{adapted} to $\fol$.
\end{definition}

We denote the module of leaf-tangent vector fields by $\vct(\fol)$ and put $T_x\fol:=T_x\fol(x)$. The \emph{dimension} and \emph{codimension} of $\fol$ are
\[\dim(\fol)=\max_{x\in M}\dim \fol(x) \quad \text{and} \quad \codim(\fol)=\dim M-\dim(\fol).\]
respectively. The connected components of the loci of leaves of a fixed dimension form the \emph{strata} $\Sigma_\alpha$ of $\fol$. Each $\Sigma_\alpha$ is an embedded submanifold, and $\fol_\alpha:=\fol|_{\Sigma_\alpha}$ is a regular Riemannian foliation with the induced metric. The regular locus
\[M_{\reg}:=\{x\in M\mid\dim \fol(x)=\dim(\fol)\}\]
is open and dense. Since the singular strata have codimension at least two, $M_{\reg}$ is also connected; see \cite[Section~6.2]{molino}. Here regularity concerns only leaf dimension, so exceptional leaves of maximal dimension are included in $M_{\reg}$.

A smooth map between foliated manifolds is \emph{foliated} if it sends each leaf into a leaf. A vector field $X$ is \emph{foliated} when its local flow consists of foliated maps. Equivalently, $[X,Y]\in\vct(\fol)$ for every $Y\in\vct(\fol)$. The foliated fields form a Lie algebra $\mathfrak L(\fol)$, and $\vct(\fol)$ is an ideal in it. The elements of
\[
\trs(\fol):=\mathfrak L(\fol)/\vct(\fol)
\]
are the \emph{transverse vector fields}. We denote the class of $X\in\mathfrak L(\fol)$ by $\overline X$. We use the same notation for the restricted foliation on any open set.

Let us recall the transverse metric in the regular case (i.e., when $M=M_{\reg}$). The tangent spaces of the leaves then form a smooth vector subbundle $T\fol\subset TM$. For $v\in T_xM$, denote by $v^\perp$ its orthogonal projection onto $(T_x\fol)^{\perp_\metric}$. The \emph{transverse metric}
\[\metric^\intercal_x(v,w):=\metric_x(v^\perp,w^\perp)\]
is smooth, has kernel $T\fol$, and is invariant under the flows of leaf-tangent fields. Consequently, the expression
\[\mathcal L_{\overline X}\metric^\intercal:=\mathcal L_X\metric^\intercal\]
is independent of the foliated representative $X$.

\begin{definition}[Transverse Killing field]
For a regular Riemannian foliation, a transverse field $\overline X$ is \emph{Killing} if $\mathcal L_{\overline X}\metric^\intercal=0$. A \emph{transverse Killing field} for a singular Riemannian foliation is a transverse field whose restriction to each stratum $\Sigma_\alpha$ is transverse Killing for $\fol_\alpha$.
\end{definition}

Notice that a foliated local flow preserves leaf dimension and therefore preserves each stratum locally, so these restrictions are well defined.

Let $\fol$ be a regular foliation of a manifold $M$. A \emph{simple open set} is a connected open subset $W\subset M$ admitting a surjective submersion $p\colon W\to S$ onto a connected smooth manifold, whose fibers are connected and are exactly the leaves of $\fol|_W$. We call $p$ a \emph{local defining submersion} and $S$ its \emph{local quotient}. The leaves of $\fol|_W$ are also called the \emph{plaques} in $W$. Every point of a regular foliation has arbitrarily small simple open neighborhoods \cite[Chapter~1]{moerdijkmrcun}. For a Riemannian foliation, the local quotient has a unique metric $\mathrm{h}$ satisfying $\metric^\intercal|_W=p^*\mathrm{h}$, so $p$ is a Riemannian submersion. A foliated field projects through $p$, by connectedness of its fibers, and its transverse class is Killing exactly when the projected field is Killing for $\mathrm{h}$ \cite[Proposition~2.1]{alexdesing}.

The normal bundle of a regular foliation $(N,\fol)$ is $\nu\fol=TM/T\fol$, and it is \emph{transversely orientable} if this bundle is orientable. The \emph{transverse orientation cover} has over $x\in M$ the two orientations of $\nu_x\fol$. The lifted foliation has its tautological transverse orientation. If the normal bundle is already orientable, either component of this cover maps diffeomorphically onto $M$. If necessary, taking also the ambient orientation cover gives a connected finite normal cover on which both orientations are available. More generally, if $p\colon \widehat M\to M$ is a connected finite normal covering, with deck group $\mathrm{Deck}(p)$, consider the pullback foliation $\widehat\fol=p^*\fol$ endowed with the pullback metric. Then $\widehat\fol$ is complete when $\fol$ is complete, and a leaf of $\widehat\fol$ is closed if and only if its image is closed. Moreover, pullback identifies $\Omega(\fol)$ with $\Omega(\widehat\fol)^{\mathrm{Deck}(p)}$ and induces an injection $p^*\colon\Ho(\fol)\longrightarrow\Ho(\widehat\fol)$. If $\fol$ is Killing, so is $\widehat\fol$, with the same structural algebra $\mathfrak a$, and the lifted structural action commutes with $\mathrm{Deck}(p)$.

Back to the general singular setting. A differential form $\omega\in\Omega^k(M)$ is \emph{basic} if $\iota_Y\omega=0$ and $\mathcal L_Y\omega=0$ for all $Y\in\vct(\fol)$. The basic forms form a differential graded algebra $\Omega(\fol)$, since the exterior derivative preserves these conditions. Its cohomology is the \emph{basic cohomology} $\Ho(\fol)$. When these groups are finite dimensional (for instance, when $M$ is compact, see \cite[Theorem 4.4]{caramellosingularequivariant}), their alternating dimension sum is the \emph{basic Euler characteristic} $\chi(\fol)$. In this case we also write $b_j(\fol)=\dim\Ho^j(\fol)$ for the \textit{basic Betti numbers} of $\fol$.

For $x\in M$, the normal space to the leaf is
\[
\nu_x\fol(x):=(T_x\fol(x))^{\perp_\metric}\subset T_xM.
\]
Choose a small plaque through $x$ and a distinguished tubular neighborhood of it. Its slice at $x$ is the image under $\exp_x^\perp$ of a small ball in $\nu_x\fol(x)$. Intersecting the local leaves with this slice, taking connected components, and pulling back by $\exp_x^\perp$ gives a foliation of that ball. Molino's homothetic lemma allows this foliation to be extended by positive dilations to the whole of $\nu_x\fol(x)$ \cite[Lemma~6.2]{molino}. The resulting singular Riemannian foliation, for the Euclidean metric $\metric_x$, is the \emph{infinitesimal foliation} $\fol_x$.

\begin{definition}[Infinitesimal polarity]
The foliation $\fol$ is \emph{infinitesimally polar} if every $\fol_x$ is polar, meaning that it admits a subspace $S<\nu_x\fol(x)$ meeting every leaf of $\fol_x$ orthogonally. Such a subspace is an \emph{infinitesimal section} at $x$.
\end{definition}

Put $q=\codim(\fol)$. Notice that every infinitesimal section has dimension $q$. When $x\in M_{\reg}$, the foliation $\fol_x$ is the foliation by points, and its only section is the whole normal space $\nu_x\fol(x)$. At a singular point, the normal space can have dimension larger than $q$, and it may contain several infinitesimal sections. By \cite[Theorem~1.4]{lytchakthorbergsson}, infinitesimal polarity at $x$ is equivalent to local closedness and the existence of a sufficiently small local Riemannian orbifold quotient near $x$.

For a complete singular Riemannian foliation, let $\overline{\fol}$ denote the partition into leaf closures. The closure theorem of Alexandrino--Radeschi \cite{alexradeschiclosure} makes $\overline{\fol}$ a singular Riemannian foliation. We say that $\fol$ is \emph{closed} when $\overline{\fol}=\fol$. In general, completeness equips $M/\overline{\fol}$ with the quotient length metric (compare the metric-quotient discussion preceding \cite[Theorem~1.4]{lytchakthorbergsson}). We define the \textit{transverse diameter} of $\fol$ by
\[\diam^{\intercal}(\fol)=\diam(M/\overline{\fol}).\]
For a complete regular Riemannian foliation this agrees with the transverse diameter defined using lengths in local quotients \cite[Proposition~3.1]{spherecaramelloneubauer}.

If $\fol$ is infinitesimally polar and closed, its leaf space has the canonical Riemannian
orbifold structure given by the transverse quotient charts and
holonomy. We denote this leaf orbifold by $M\Slash\fol$, whose underlying
space is $M/\fol$. Euler characteristics of orbifolds refer to their underlying spaces, or equivalently to their real de Rham cohomology.

Let us now briefly recall Molino's structural theory. Let $\mathscr{C}$ be a sheaf of Lie algebras on $M$. The \emph{stalk} $\mathscr C_x$ of a sheaf consists of germs of local sections at $x$, with two sections defining the same germ when they agree near $x$. For a finite-dimensional Lie algebra $\mathfrak h$, the constant sheaf $\underline{\mathfrak h}_M$ consists of locally constant maps to $\mathfrak h$. A sheaf is \emph{locally constant} if it admits such an identification near every point, and \emph{globally constant} if one identification exists on all of $M$.

Molino's second structural theorem associates to a complete regular Riemannian foliation $\fol$ a canonical locally constant sheaf $\mathscr C_{\fol}$ of Lie algebras of germs of local transverse Killing fields \cite[Theorem~5.2]{molino}. Its local sections describe the leaf closures, in the following sense: if $\overline{X_1},\ldots,\overline{X_r}$ form a local basis of the sheaf on a sufficiently small neighborhood $U$, then
\[T_y\overline{\fol(y)}=T_y\fol(y)+\spannn\{X_1(y),\ldots,X_r(y)\}\]
for each $y\in U$. The typical stalk is isomorphic to the opposite of the structural Lie algebra $\mathfrak{g}_\fol:=\mathfrak{l}(\fol|_L)$ of $\fol$.

In the singular case, the Molino sheaf of $\fol|_{M_{\reg}}$ extends continuously to a locally constant sheaf $\mathscr C_{\fol}$ on $M$ \cite[Lemma~6.5]{molino}. Its sections restrict to transverse Killing fields on each stratum. We say that $\mathscr C_{\fol}$ is \emph{smooth} if every such section admits, locally, a smooth foliated representative on $M$, including at singular points. This condition concerns the realization of the sheaf by vector fields; local constancy alone is a property of its abstract local systems of Lie algebras. The structural Lie algebra of $\fol$ is that of $\fol|_{M_{\reg}}$.

The assertion that this smoothness holds for every complete singular Riemannian foliation is the \emph{strong Molino conjecture} \cite[p.~215]{molino}; see \cite[Section~9.4]{surveyalexcaramello} and \cite[Section~3.1]{caramellosingularequivariant} for this terminology. Here we establish it for the infinitesimally polar class in Theorem~\ref{theorem strong molino infinitesimally polar}.

\begin{definition}[Killing foliation]
A complete regular Riemannian foliation is a \emph{Killing foliation} if its Molino sheaf is globally constant. A complete singular Riemannian foliation is a \emph{singular Killing foliation} if its Molino sheaf is smooth and globally constant (as in \cite{blupcaramello}).
\end{definition}

Smoothness remains part of the definition of a singular Killing foliation. For complete infinitesimally polar foliations, however, it is automatic by Theorem~\ref{theorem strong molino infinitesimally polar}. In this class, global constancy of the Molino sheaf is therefore equivalent to the Killing condition. When the foliation is Killing, its structural Lie algebra is Abelian and is denoted by $\mathfrak{a}$. The corresponding global Molino fields define its natural transverse $\mathfrak{a}$-action $\actionarrow{\mathfrak{a}}{\fol}$.

For a regular Riemannian foliation $\fol$, let $\operatorname{iso}(\fol)$ be its Lie algebra of global transverse Killing fields and define its \emph{transverse symmetry rank} by
\[
\symrank(\fol)
=\max\{\dim\mathfrak t \mid \mathfrak t\subset\operatorname{iso}(\fol)
\text{ is an Abelian Lie subalgebra}\}.
\]
For a closed foliation on a compact manifold, this is the symmetry rank of its compact leaf orbifold.

\section{Lifting fields to the geometric resolution}
\label{section lifting}

We start by recalling Lytchak's geometric resolution. For a vector bundle $E\to M$, its \emph{Grassmannian bundle of $q$-planes} is
\[\Gr_q(E):=\{(x,P)\mid x\in M,\ P\subset E_x\text{ is a linear subspace with }\dim P=q\}.\]
Its projection sends $(x,P)$ to $x$, and its smooth structure is obtained from vector bundle trivializations and the usual smooth Grassmannians of the fibers.

Suppose now that $\fol$ is infinitesimally polar. Lytchak constructs the resolution as the subset
\[\Gr M:=\{(x,S)\in\Gr_q(TM)\mid S\text{ is an infinitesimal section at }x\},\]
with projection $\rho(x,S)=x$. We will use the following properties of this construction \cite[Theorem~1.1 and Sections~4.2--4.6]{lytchak}:
\begin{itemize}
\item $\Gr M$ is a closed smooth embedded submanifold of $\Gr_q(TM)$, with $\dim\Gr M=\dim M$, and $\rho$ is proper and surjective.
\item $\Gr M$ is connected, since it contains the connected dense subset $M_{\reg}^{\Gr}\cong M_{\reg}$.
\item The sets $\rho^{-1}(L)$, for $L\in\fol$, are the leaves of a regular foliation $\Gr\fol$ of codimension $q$.
\item The open subset $M_{\reg}^{\Gr}:=\rho^{-1}(M_{\reg})$ is dense in $\Gr M$, and the restriction
\[
\rho_{\reg}:=\rho|_{M_{\reg}^{\Gr}}\colon M_{\reg}^{\Gr}\longrightarrow M_{\reg}
\]
is a foliated diffeomorphism.
\item There is a Riemannian metric $\Gr\metric$ adapted to $\Gr\fol$ for which $\rho$ preserves transverse lengths. In particular, writing $\metric_{\reg}^\intercal$ for the transverse metric of $\fol|_{M_{\reg}}$, one has
\begin{equation}\label{equation transverse metrics}
(\Gr\metric)^\intercal|_{M_{\reg}^{\Gr}}
=\rho_{\reg}^*\metric_{\reg}^\intercal.
\end{equation}
\item If $\metric$ is complete (compact), then $\Gr\metric$ is complete (compact).
\end{itemize}

\begin{lemma}\label{lemma quotient comparison applications}
Let $(M,\fol,\metric)$ be complete and infinitesimally polar. The resolution $\rho\colon\Gr M\to M$ induces a homeomorphism $\Gr M/\Gr\fol\cong M/\fol$ and an isometry
\[\Gr M/\overline{\Gr\fol}\cong M/\overline{\fol}.\]
Moreover, for every leaf $L$ of $\fol$,
\[\overline{\rho^{-1}(L)}=\rho^{-1}(\overline L).\]
In particular, closed leaves correspond, closedness of the entire
foliation is preserved, and $\diam^{\intercal}(\Gr\fol)=\diam^{\intercal}(\fol)$. If $\fol$ is closed, the induced identification $\Gr M\Slash\Gr\fol\cong M\Slash\fol$ is an isometry of Riemannian orbifolds.
\end{lemma}

\begin{proof}
Properness and surjectivity make $\rho$ a quotient map, and its inverse images of leaves are precisely the resolved leaves. The induced leaf-space bijection is therefore a homeomorphism. Leaf projections are open, since saturations of open sets are unions of images under local leaf-tangent flows. For an open quotient map $p$, one has $\overline{p^{-1}(z)}=p^{-1}(\overline{\{z\}})$: every neighborhood of a point on the right has open image meeting $\{z\}$. Applying this to both leaf projections proves the closure identity and the correspondence of closed leaves.

The resolution preserves the local quotient length structures \cite[Theorem~1.1 and Section~2.6]{lytchak}. The transverse pseudodistance, obtained from chains of local quotient paths joined along leaves, is consequently preserved. For complete Riemannian foliations, its metric quotient is the leaf-closure space, proving the isometry and diameter equality. In the closed case, the same local transverse isometries identify the leaf orbifolds.
\end{proof}

We now give a description of $\Gr M$ inside the cotangent Grassmannian. Recall that, for a subspace $V$ of a finite-dimensional vector space $E$, its \emph{annihilator} is
\[\Ann(V)=\{\alpha\in E^* \mid \alpha|_V=0\}.\]
Thus $\Ann(T_x\fol)$ is the conormal space at a regular point and has dimension $q$. The metric isomorphism $\metric^\flat(v)=\metric(v,\cdot)$, with inverse $\metric^\sharp$, satisfies $\metric_x^\flat(V^{\perp_\metric})=\Ann(V)$ and induces the Grassmannian bundle diffeomorphism $B_{\metric}\colon\Gr_q(TM)\longrightarrow\Gr_q(T^*M)$, given by
\[\qquad B_{\metric}(x,S)=(x,\metric_x^\flat(S)).\]
Since $\Gr M$ is closed and $M_{\reg}^{\Gr}$ is dense in it, this diffeomorphism identifies the resolution with
\begin{equation}\label{equation conormal model}
(\Gr M)^*
=\overline{\{(x,\Ann(T_x\fol)):x\in M_{\reg}\}}
\subset\Gr_q(T^*M)
\end{equation}
(the closure is taken in the total Grassmannian bundle). At a singular point, $\Ann(T_x\fol)$ has dimension greater than $q$; the exceptional fiber consists of the limiting $q$-planes in \eqref{equation conormal model}, equivalently the metric images of the infinitesimal sections.

We now lift foliated diffeomorphisms and their infinitesimal generators.

\begin{prop}\label{proposition lifting diffeomorphisms}
Let $\phi\colon U\to V$ be a smooth foliated diffeomorphism between open subsets of $M$. Then $\phi$ has a canonical smooth foliated lift
\[\phi^{\Gr}\colon\rho^{-1}(U)\longrightarrow\rho^{-1}(V),\]
that is, $\rho\circ\phi^{\Gr}=\phi\circ\rho$. The lifts respect composition and depend smoothly on any smooth parameters in $\phi$.
\end{prop}

\begin{proof}
The diffeomorphism $\phi$ acts on the cotangent Grassmannian by
\begin{equation}\label{equation cotangent action}
\Phi_\phi(x,P)=\left(\phi(x),\bigl((\dif\phi_x)^{-1}\bigr)^*P\right).
\end{equation}
Since $\dif\phi_x(T_x\fol)=T_{\phi(x)}\fol$, it sends $\Ann(T_x\fol)$ onto $\Ann(T_{\phi(x)}\fol)$. Hence $\Phi_\phi$ preserves the regular conormal graph and its closure \eqref{equation conormal model}, restricted to the respective open sets. Define
\[\phi^{\Gr}=B_{\metric}^{-1}\circ\Phi_\phi\circ B_{\metric}\]
on $\rho^{-1}(U)$. Its inverse is the lift of $\phi^{-1}$, and it covers $\phi$. It is foliated because the resolved leaves are inverse images of the original leaves, with connected components taken on restricted domains. Finally, the chain rule gives $(\phi\circ\psi)^{\Gr}=\phi^{\Gr}\circ\psi^{\Gr}$, and \eqref{equation cotangent action} gives smooth dependence on parameters.
\end{proof}

We can now also litf transverse vector fields.

\begin{prop}\label{proposition lifting transverse fields}
For every open subset $U\subset M$, transverse lifting defines a Lie algebra homomorphism
\[\trs(\fol|_U)\longrightarrow\trs(\Gr\fol|_{\rho^{-1}(U)})\]
which we will denote by $\overline X\mapsto(\overline X)^{\Gr}$, that is compatible with restrictions and sends transverse Killing fields to transverse Killing fields. If $X$ is a smooth foliated representative, its lifted flow has a smooth foliated generator $X^{\Gr}$ satisfying
\begin{equation}\label{equation related fields}
\dif\rho(X^{\Gr})=X\circ\rho\quad \text{and}\quad (\overline X)^{\Gr}=\overline{X^{\Gr}}.
\end{equation}
\end{prop}

\begin{proof}
Lift the local flow $\phi_t$ of $X$ by Proposition~\ref{proposition lifting diffeomorphisms}. Functoriality makes $\phi_t^{\Gr}$ a smooth local flow, whose foliated generator $X^{\Gr}$ satisfies \eqref{equation related fields}. Over the regular locus,
\begin{equation}
X^{\Gr}_{x^{\Gr}}=(\dif\rho_{x^{\Gr}})^{-1}X_{\rho(x^{\Gr})}.
\end{equation}
Thus lifting preserves linear combinations and brackets there, and hence everywhere by smoothness and density. If $X-X'$ is leaf-tangent, the normal class of $X^{\Gr}-(X')^{\Gr}$ vanishes over the regular locus. It is a smooth section of the vector bundle $T\Gr M/T\Gr\fol$, so it vanishes everywhere. This proves independence of the representative and the Lie algebra assertion; the construction also commutes with restrictions.

For a transverse Killing field, \eqref{equation transverse metrics} and \eqref{equation related fields} give
\[\mathcal L_{X^{\Gr}}(\Gr\metric)^\intercal=\rho_{\reg}^*(\mathcal L_X\metric_{\reg}^\intercal)=0\]
over $\rho^{-1}(U\cap M_{\reg})$. The tensor on the left is smooth on the regular resolution, so it vanishes everywhere by density.
\end{proof}

\section{The strong Molino conjecture and the Molino sheaf of the resolution}
\label{section molino resolution}

We first establish the strong Molino conjecture for complete infinitesimally polar singular Riemannian foliations. Our argument combines Molino's local transverse flows with the descent of basic one-forms through the canonical resolution, a mechanism already present in \cite[Theorem~3.6 and Remark~3.1]{alexlytchak}.

Before the proof, we establish a useful fact. For a sufficiently small distinguished neighborhood $U$, write $r\colon U\to B$ for its plaque quotient, which in our setting is a Riemannian orbifold. A section $s\in\mathscr C_{\fol}(U)$ has a smooth horizontal representative on $U\cap M_{\reg}$. This representative is locally bounded, being the orthogonal projection of a continuous representative supplied by Molino's construction \cite[Lemma~6.5]{molino}. Its integral curves remain in leaf closures, which are contained in $M_{\reg}$ \cite[Lemma~6.4]{molino}.
Together with bounded speed, this ensures existence of its flow for a common sufficiently small time interval on a smaller neighborhood. The induced local isometries of the regular plaque quotient extend uniquely to $B$ by local metric completion, and the bounded-speed estimate gives joint continuity. This produces the continuous local isometric flow associated with $s$; compare \cite[proof of Theorem~3.4]{caramellosingularequivariant}. We now verify that it defines a smooth orbifold Killing field.

Choose an orbifold chart
\[q\colon(T,\mathrm{h}_T)\longrightarrow T/\Gamma\subset B.\]
After restricting domains and times if necessary, the local quotient flow is smooth, by \cite[Theorem~3.1]{alexradeschiflows}. Schwarz's isotopy lifting theorem \cite[Corollary~2.4]{schwarzlifting} then provides a smooth $\Gamma$-equivariant local lift $\widehat\varphi_t^s$, starting at the identity. On the dense open subset of $T$ where $\Gamma$ acts freely, $q$ is a local isometry. Hence $\widehat\varphi_t^s$ is a local isometry there, and
smoothness gives $(\widehat\varphi_t^s)^*\mathrm{h}_T=\mathrm{h}_T$ throughout its domain. Its infinitesimal generator $\widehat Y_s$ is therefore a smooth $\Gamma$-invariant Killing field. The chart generators agree under changes of uniformizing charts: they agree over the dense free loci, and hence everywhere by smoothness. They consequently define a smooth orbifold Killing field $Y_s$ on $B$, representing $s$ on the regular plaque quotient.

\begin{thm}[Strong Molino conjecture for infinitesimally polar foliations]\label{theorem strong molino infinitesimally polar}
Let $(M,\fol,\metric)$ be a complete infinitesimally polar singular Riemannian foliation. Then its Molino sheaf is smooth. More precisely, for every open set $V\subset M$ and every section $s\in\mathscr C_{\fol}(V)$, there is a unique smooth horizontal foliated vector field $X_s$ representing $s$ on $V$.
\end{thm}

\begin{proof}
Fix $x\in V$ and choose a sufficiently small distinguished neighborhood $U\subset V$. Write $r\colon U\to B$ for its plaque quotient and $\mathrm{h}$ for the quotient metric. As we saw above, infinitesimal polarity makes $B$ a Riemannian orbifold, and the local flow of $s|_U$ has a smooth orbifold Killing generator $Y_s$. Hence $\beta_s=\mathrm{h}(Y_s,\cdot)$ is a smooth orbifold one-form.

Now put $U^{\Gr}=\rho^{-1}(U)$ and define
\[r^{\Gr}:=r\circ\bigl(\rho|_{U^{\Gr}}\bigr)\colon U^{\Gr}\longrightarrow B.\]
The resolution preserves the local transverse orbifold charts and their metrics (see \cite[Theorem~1.1 and Sections~4.4--4.5]{lytchak}) so, on sufficiently small simple open sets $W\subset U^{\Gr}$, we have
$r^{\Gr}|_W=q_W\circ p_W$, where $p_W\colon W\to T_W$ is a defining Riemannian submersion
and $q_W$ is the corresponding orbifold chart projection. Writing $\beta_{s,W}$ for the chart representative of
$\beta_s$, the forms $p_W^*\beta_{s,W}$ agree on overlaps
and define a smooth basic one-form
\[\alpha_s^{\Gr}:=(r^{\Gr})^*\beta_s\in\Omega^1(\Gr\fol|_{U^{\Gr}}).\]
On each defining submersion, the dual $X^{\Gr}_s=(\alpha_s^{\Gr})^{\sharp_{\Gr\metric}}$ is the horizontal lift of the chart generator, thus it represents a transverse Killing field.

By \cite[Proposition~3.5]{alexlytchak} (here also recalled in Proposition~\ref{proposition basic forms resolution}), there is a unique smooth basic one-form $\alpha_s$ on $U$ such that $\rho^*\alpha_s=\alpha^{\Gr}_s$. Define
\begin{equation}
X_s=\alpha_s^{\sharp_{\metric}}=\left[(\rho^*)^{-1}\bigl(\Gr\metric(X^{\Gr}_s,\cdot)\bigr)\right]^{\sharp_{\metric}}.
\end{equation}
The field $X_s$ is hence smooth on all of $U$, and it is horizontal because $\alpha_s$ annihilates every leaf-tangent vector.

By the construction of $\Gr\metric$, the differential of $\rho_{\reg}$ maps horizontal vectors isometrically
onto horizontal vectors \cite[Sections~4.4--4.5]{lytchak}. Together with $\rho^*\alpha_s=\alpha_s^{\Gr}$, this
shows that, over $U\cap M_{\reg}$, $X_s$ is the horizontal representative of $s$ and is $\rho$-related to $X_s^{\Gr}$. By smoothness and density, this relation extends to
\begin{equation}
\dif\rho(X_s^{\Gr})=X_s\circ\rho
\end{equation}
on $U^{\Gr}$. Uniqueness of integral curves makes their local flows commute with $\rho$. Properness makes each fiber compact, so the comparison holds for a common small time interval along the whole fiber. The flow upstairs preserves $\Gr\fol$; since resolved leaves map onto original leaves, the flow downstairs preserves $\fol$ and its strata. Its quotient flow is the prescribed isometric flow on $B$. Hence $X_s$ represents $s|_U$ and is transverse Killing on every stratum.

Finally, two smooth horizontal representatives of the same section agree on the regular locus, where a transverse class has a unique horizontal representative. So they agree everywhere by density. The representatives constructed on neighborhoods of the points of $V$ therefore agree on overlaps and glue to a smooth horizontal representative on $V$. The same argument proves its uniqueness.
\end{proof}

We now identify the Molino sheaves by lifting their sections through Proposition~\ref{proposition lifting transverse fields}.

\begin{thm}\label{theorem molino resolution}
Let $(M,\fol,\metric)$ be a complete infinitesimally polar singular Riemannian foliation. Then its geometric resolution has a canonical isomorphism of sheaves of Lie algebras
\begin{equation}\label{equation molino isomorphism}
\rho^{-1}\mathscr C_{\fol}\cong\mathscr C_{\Gr\fol},
\end{equation}
induced by the transverse lifting of Proposition~\ref{proposition lifting transverse fields}. In particular,
$\mathfrak g_{\fol}\cong\mathfrak g_{\Gr\fol}$.
\end{thm}

\begin{proof}
The regular locus $M_{\reg}$ is an open union of leaf closures \cite[Lemma~6.4]{molino}. The same holds for $M_{\reg}^{\Gr}$: for $z\in M_{\reg}^{\Gr}$, continuity of $\rho$ gives
\[\rho\bigl(\overline{\Gr\fol(z)}\bigr)\subset\overline{\fol(\rho(z))}\subset M_{\reg}.\]
Therefore, taking the Molino sheaf commutes with restriction to such open sets. The foliated transverse isometry $\rho_{\reg}$ induces an isomorphism between the transverse orthonormal frame bundles of $\fol_{\reg}$ and $\Gr\fol|_{M_{\reg}^{\Gr}}$, preserving their canonical transverse parallelisms (see \cite[Section~5.3]{molino} for details). It therefore identifies their commuting sheaves and the projected Molino sheaves. Since $\mathscr C_{\fol}|_{M_{\reg}}=\mathscr C_{\fol_{\reg}}$ (by definition of $\mathscr C_{\fol}$ as an extension), we obtain
\begin{equation}\label{equation regular molino sheaves}
\mathscr C_{\Gr\fol}|_{M_{\reg}^{\Gr}} \cong (\rho^{-1}\mathscr C_{\fol})|_{M_{\reg}^{\Gr}}.
\end{equation}
On this locus, the identification is precisely the transverse lifting of Proposition~\ref{proposition lifting transverse fields}, because $\rho_{\reg}$ is a foliated diffeomorphism.

Fix $x^{\Gr}\in\Gr M$ and choose a connected neighborhood $U$ of $\rho(x^{\Gr})$ trivializing $\mathscr C_{\fol}$, with basis $s_1,\ldots,s_r$. Theorem~\ref{theorem strong molino infinitesimally polar} and Proposition~\ref{proposition lifting transverse fields} give smooth transverse Killing lifts $s_i^{\Gr}$ on $\rho^{-1}(U)$. Choose a simple open set $W\subset\rho^{-1}(U)$ containing $x^{\Gr}$ and trivializing $\mathscr C_{\Gr\fol}$, and choose $y^{\Gr}\in W\cap M_{\reg}^{\Gr}$. By \eqref{equation regular molino sheaves}, the germs of $s_i^{\Gr}$ there form a basis of the upstairs Molino stalk. They extend uniquely to a basis $t_1,\ldots,t_r$ of that constant sheaf on $W$.

The transverse Killing fields $t_i$ and $s_i^{\Gr}|_W$ agree near $y^{\Gr}$. They project to Killing fields on the connected local quotient of $W$, where uniqueness from their values and first derivatives at a point (see e.g. \cite[Proposition~8.1.4]{petersen}) makes them agree throughout $W$. Thus the lifted germs belong to the Molino sheaf and form a basis also at $x^{\Gr}$. Lifting therefore gives isomorphisms on all stalks. Its compatibility with restrictions and independence of representatives glue these into \eqref{equation molino isomorphism}, and its bracket compatibility makes this an isomorphism of Lie algebra sheaves.
\end{proof}

\begin{cor}\label{corollary killing resolution}
Let $(M,\fol,\metric)$ be a complete infinitesimally polar singular Killing foliation. Then $\Gr\fol$ is a complete regular Killing foliation with the same structural Lie algebra $\mathfrak{a}$. Its natural transverse $\mathfrak{a}$-action is the lift of $\actionarrow{\mathfrak{a}}{\fol}$.
\end{cor}

\begin{proof}
Theorem~\ref{theorem molino resolution} gives $\mathscr C_{\Gr\fol}\cong\rho^{-1}\underline{\mathfrak a}_M
\cong\underline{\mathfrak a}_{\Gr M}$. The transverse lifting from Proposition \ref{proposition lifting transverse fields} identifies the structural algebras and actions.
\end{proof}

\begin{cor}\label{corollary killing criterion infinitesimally polar}
A complete infinitesimally polar singular Riemannian foliation is Killing if and only if its Molino sheaf is globally constant. In particular, every such foliation on a simply connected manifold is Killing.
\end{cor}

\begin{proof}
Smoothness of $\mathscr C_\fol$ is automatic by Theorem~\ref{theorem strong molino infinitesimally polar}. On a connected simply connected manifold, a locally constant sheaf is globally constant, since continuation of germs along paths has trivial monodromy.
\end{proof}

\section{Basic cohomology and transverse curvature through resolutions}
\label{section cohomology curvature}

Let $(M,\fol,\metric)$ be an infinitesimally polar singular Riemannian foliation, and let
\[\rho\colon(\Gr M,\Gr\fol,\Gr\metric)\longrightarrow(M,\fol,\metric)\]
be its geometric resolution. As before, we write $q=\codim(\fol)$, $M_{\reg}$ for the regular locus, and $M_{\reg}^{\Gr}=\rho^{-1}(M_{\reg})$. In this section we recall the preservation of basic forms and compare the transverse curvatures of the two foliations. Let us start by recalling Alexandrino--Lytchak's preservation of basic forms in the form needed below.

\begin{prop}[{\cite[Proposition~3.5]{alexlytchak}}]\label{proposition basic forms resolution}
For every open subset $U\subset M$, pullback induces an isomorphism of differential graded algebras
\[\rho^*\colon\Omega(\fol|_U)\longrightarrow\Omega(\Gr\fol|_{\rho^{-1}(U)}),\]
which commutes with restriction to open subsets. In particular, the induced map $\rho^*\colon\Ho(\fol)\longrightarrow\Ho(\Gr\fol)$ is also an isomorphism of graded algebras.
\end{prop}

It therefore follows that, if $\Ho(\fol)$ is finite dimensional, then
\[\chi(\fol)=\chi(\Gr\fol).\]

Let us now compare equivariant cohomologies. We use the basic Cartan model of \cite[Section~5]{caramellosingularequivariant}. A local foliated representative $X_A$ of $A\in\mathfrak a$ acts on basic forms by $\iota_A=\iota_{X_A}$ and $\mathcal L_A=\mathcal L_{X_A}$. These operators preserve basicness and are independent of representatives, so they agree on overlaps. With linear polynomials in $S(\mathfrak a^*)$ placed in degree two, the Cartan complex is
\[C_{\mathfrak a}(\fol)=\bigl(S(\mathfrak a^*)\otimes\Omega(\fol)\bigr)^{\mathfrak a}\]
with equivariant differential $(\dif_{\mathfrak a}\omega)(A)=\dif(\omega(A))-\iota_A\omega(A)$. Here $\omega$ is a polynomial map to basic forms; invariance and Cartan's identity give $\dif_{\mathfrak a}^2=0$. Its cohomology is $\Ho_{\mathfrak a}(\fol)$. The structural action $\actionarrow{\mathfrak{a}}{\fol}$ is \emph{equivariantly formal} when $\Ho_{\mathfrak a}(\fol)\cong S(\mathfrak a^*)\otimes\Ho(\fol)$ as graded $S(\mathfrak a^*)$-modules.

\begin{prop}\label{proposition equivariant cohomology resolution}
Pullback induces an isomorphism of graded $S(\mathfrak{a}^*)$-algebras
\[\rho^*\colon\Ho_{\mathfrak{a}}(\fol)\longrightarrow\Ho_{\mathfrak{a}}(\Gr\fol).\]
Moreover, $\actionarrow{\mathfrak{a}}{\fol}$ is equivariantly formal if and only
if $\actionarrow{\mathfrak{a}}{\Gr\fol}$ is equivariantly formal.
\end{prop}

\begin{proof}
Proposition~\ref{proposition basic forms resolution} identifies the basic differential graded algebras. The lifted structural fields are $\rho$-related to the original representatives, so naturality gives
\[\rho^*\iota_A=\iota_A\rho^*\quad\text{and}\quad \rho^*\mathcal L_A=\mathcal L_A\rho^*.\]
Thus $1\otimes\rho^*$ identifies the invariant Cartan complexes, preserving their differentials and polynomial module structures. Passing to cohomology proves the isomorphism. The preservation of ordinary basic cohomology then gives the formality assertion.
\end{proof}

We now study the behavior of transverse curvature through geometric resolution. We define transverse curvature through the local quotients, equivalently through the basic connection \cite[Section~2.3]{positivecaramellotoben}. Let $(M,\fol,\metric)$ first be regular, and put $H=(T\fol)^{\perp_\metric}$. On a simple open set $W$, let $\pi\colon(W,\metric)\to(B,\mathrm{h}_B)$ be a defining Riemannian submersion. The \emph{transverse curvature} and \emph{transverse Ricci tensors} are
\[R_{\fol}^{\intercal}=\pi^*R^{\mathrm{h}_B}\quad \text{and}\quad \ric_{\fol}^{\intercal}=\pi^*\ric_{\mathrm{h}_B},
\]
where $R^{\mathrm{h}_B}$ denotes the covariant $(0,4)$-curvature tensor of $B$. Changes of local quotient are local isometries, so these tensors agree on overlaps and vanish whenever an argument is leaf-tangent. For a horizontal two-plane $\sigma=\spannn\{u,v\}\subset H_x$, the respective transverse sectional curvature is
\[\sec_{\fol}(x,\sigma)=\frac{R_{\fol}^{\intercal}(u,v,v,u)}{\metric(u,u)\metric(v,v)-\metric(u,v)^2}.\]
For a singular foliation, these definitions apply on $M_{\reg}$, with horizontal bundle $H_{\reg}$. Upstairs, the horizontal bundle $H^{\Gr}=(T\Gr\fol)^{\perp_{\Gr\metric}}$ and the transverse tensors are smooth on all of $\Gr M$. Numerical Ricci bounds are evaluated on unit horizontal vectors: $\ric_{\fol}^{\intercal}\geq\lambda$ means $\ric_{\fol}^{\intercal}(v,v)\geq\lambda$ for $\metric(v,v)=1$. 

\begin{thm}\label{theorem transverse curvature resolution}
On $M_{\reg}^{\Gr}$, the transverse curvature and Ricci tensors satisfy
\begin{equation}\label{equation transverse curvature pullback}
R_{\Gr\fol}^{\intercal}|_{M_{\reg}^{\Gr}}=\rho_{\reg}^*R_{\fol}^{\intercal}\quad \text{and}\quad\ric_{\Gr\fol}^{\intercal}|_{M_{\reg}^{\Gr}}=\rho_{\reg}^*\ric_{\fol}^{\intercal}.
\end{equation}
If $q\geq2$, then for every $x^{\Gr}\in M_{\reg}^{\Gr}$ and every two-plane $\sigma^{\Gr}\subset H^{\Gr}_{x^{\Gr}}$,
\begin{equation}\label{equation transverse sectional equality}
\sec_{\Gr\fol}(x^{\Gr},\sigma^{\Gr})=\sec_{\fol}\bigl(\rho(x^{\Gr}),\dif\rho_{x^{\Gr}}(\sigma^{\Gr})\bigr).
\end{equation}
\end{thm}

\begin{proof}
The tensors on the left of \eqref{equation transverse curvature pullback} are the unique smooth extensions of the displayed pullbacks to $\Gr M$. Choose a simple open set $U\subset M_{\reg}$ with defining Riemannian submersion $\pi\colon(U,\metric)\to(B,\mathrm{h}_B)$. Lytchak's metric makes $\dif\rho$ an isometry from the resolved horizontal space onto the recorded infinitesimal section \cite[Sections~4.4--4.5]{lytchak}. Over the regular locus, this section is $H_{\reg}$, so $\pi^{\Gr}=\pi\circ\rho_{\reg}$ is a defining Riemannian submersion on $\rho^{-1}(U)$ onto the same $(B,\mathrm{h}_B)$. Pulling back its curvature and Ricci tensors proves \eqref{equation transverse curvature pullback}; the horizontal isometry also proves \eqref{equation transverse sectional equality}. Smoothness of the resolved tensors and density of $M_{\reg}^{\Gr}$ give its uniqueness.
\end{proof}

At a singular point $x^{\Gr}=(x,S)$, the isometry $\dif\rho|_{H^{\Gr}_{x^{\Gr}}}\colon H^{\Gr}_{x^{\Gr}}\to S$ transports the resolved tensor to the $q$-dimensional infinitesimal section $S$. Its values are limits of regular transverse curvatures along sequences converging to $(x,S)$. Notice that the section data are necessary: the full normal space at $x$ can have dimension greater than $q$, and this tensor is not the transverse curvature of the foliation restricted to a singular stratum.

\begin{cor}\label{corollary transverse curvature bounds resolution}
Let $a,b,c,d\in\mathbb R$. If $q\geq2$, then
\[a\leq\sec_{\fol}\leq b\quad\text{on }M_{\reg}\quad\Longleftrightarrow\quad a\leq\sec_{\Gr\fol}\leq b\quad\text{on }\Gr M.\]
For every positive codimension, one also has
\[c\leq\ric_{\fol}^{\intercal}\leq d\quad \text{on }M_{\reg}\quad\Longleftrightarrow\quad c\leq\ric_{\Gr\fol}^{\intercal}\leq d \quad\text{on }\Gr M.\]
\end{cor}

\begin{proof}
Theorem~\ref{theorem transverse curvature resolution} gives the inequalities over $M_{\reg}^{\Gr}$. Continuity of sectional curvature and of $v\mapsto\ric_{\Gr\fol}^{\intercal}(v,v)$ extends the bounds everywhere. The converses are clear by restriction.
\end{proof}

A strict bound with a uniform margin therefore remains strict upstairs. Constant sectional curvature and the transverse Einstein identity $\ric_{\fol}^{\intercal}=\lambda\metric_{\reg}^{\intercal}$ also persist through the resolution.

\section{Applications to infinitesimally polar singular Killing foliations}
\label{section applications}

Throughout this section, $\fol$ is an infinitesimally polar singular Killing foliation on $M$, with structural algebra $\mathfrak a$, codimension $q$ and defect $d=\dim\mathfrak a$, except where explicitly stated otherwise. Curvature bounds refer to $M_{\reg}$, and sectional curvature statements assume $q\geq2$.

We apply the regular results of \cite{positivecaramellotoben,equivariantcaramellotoben,spherecaramelloneubauer} to the geometric resolution, using the comparisons of Section~\ref{section cohomology curvature} and Corollary~\ref{corollary killing resolution}. The quotient and covering facts needed below were recalled in Section~\ref{section background}.

\begin{thm}\label{theorem closed orbifold models applications}
Suppose that $M$ is compact. There is a compact $q$-dimensional
Riemannian orbifold $\mathcal O$ with an effective isometric
$\mathbb{T}^d$-action such that
\[
M/\overline{\fol}\cong|\mathcal O|/\mathbb{T}^d,
\qquad
\chi(\fol)=\chi(|\mathcal O|).
\]
Here $|\mathcal O|$ is the underlying topological space, and the
first identification is a homeomorphism. There is an isomorphism
of graded $\RR$-algebras
\[
\Ho_{\mathfrak{a}}(\fol)\cong\Ho_{\mathbb{T}^d}(\mathcal O).
\]
If $\actionarrow{\mathfrak{a}}{\fol}$ is equivariantly formal, then $\actionarrow{\mathbb{T}^d}{\mathcal O}$ is equivariantly formal, and
\[
b_j(\fol)=b_j(\mathcal O)=b_j(|\mathcal O|)
\quad\text{for every }j.
\]
\end{thm}

\begin{proof}
Properness makes $\Gr M$ compact. Apply
\cite[Theorem~B]{positivecaramellotoben} to its regular Killing
foliation. It gives a deformation $(\Gr\fol)_t$, $0\leq t\leq1$,
with $(\Gr\fol)_0=\Gr\fol$ and a closed endpoint
$(\Gr\fol)_1=\Gr\mathcal{G}$ that can be chosen
arbitrarily close to $\Gr\fol$ (as sections of the respective Grasmannian bundle). The deformation stays inside
the original leaf closures, and the transverse metric deforms
smoothly with it. Set $\mathcal O=\Gr M\Slash\Gr\mathcal{G}$.
The induced structural action at the endpoint integrates to an
effective isometric $\mathbb{T}^d$-action on $\mathcal O$.
The leaf-closure quotient in the cited theorem is $M/\overline{\fol}$ by
Lemma~\ref{lemma quotient comparison applications}.
The equality of Euler characteristics follows from
Proposition~\ref{proposition basic forms resolution},
\cite[Theorem~7.4]{positivecaramellotoben}, and the de Rham theorem for orbifolds, as recalled in \cite[Theorem~2.2 and Proposition~2.3]{equivariantcaramellotoben}.

Proposition~\ref{proposition equivariant cohomology resolution},
together with \cite[Proposition~5.2, Theorem~5.3 and
Corollary~5.5]{equivariantcaramellotoben}, gives the equivariant
isomorphisms through the whole deformation,
\[
\Ho_{\mathfrak{a}}(\fol)
\cong\Ho_{\mathfrak{a}}((\Gr\fol)_t)
\cong\Ho_{\mathbb{T}^d}(\mathcal O).
\]
Here the action on $(\Gr\fol)_t$ is the induced action of the
deformation construction. If the initial action is equivariantly
formal, we can use
\cite[Proposition~6.2 and Theorem~6.3]{equivariantcaramellotoben}
to obtain formality and constancy of the basic Betti numbers.
Finally, the basic cohomology of the closed foliation
$\Gr\mathcal{G}$ is the real cohomology of $\mathcal O$,
which agrees with that of $|\mathcal O|$.
\end{proof}

The isomorphisms through that deformation are isomorphisms of $\RR$-algebras; their $S(\mathfrak{a}^*)$-module structures can change
\cite[Remark~5.4]{equivariantcaramellotoben}. The pullback isomorphism induced by $\rho$ (see Proposition~\ref{proposition equivariant cohomology resolution}), in contrast, preserves this module structure.

\begin{cor}\label{corollary Euler localization applications}
Suppose that $M$ is compact, and let $M^{\mathfrak{a}}\subset M$ be the union of
the closed leaves. Then
\[
\chi(\fol)=\chi(M^{\mathfrak{a}}/\fol).
\]
In particular, $\chi(\fol)\neq0$ implies that $\fol$ has a closed leaf.
\end{cor}

\begin{proof}
By Lemma~\ref{lemma quotient comparison applications},
$(\Gr M)^{\mathfrak{a}}=\rho^{-1}(M^{\mathfrak{a}})$ is the union of the closed leaves of
$\Gr\fol$ and
$(\Gr M)^{\mathfrak{a}}/\Gr\fol\cong M^{\mathfrak{a}}/\fol$. We can hence apply \cite[Theorem~D]{positivecaramellotoben} upstairs and use $\chi(\fol)=\chi(\Gr\fol)$.
\end{proof}

We next turn to positive curvature, beginning with two Euler characteristic observations for regular foliations. We first record a low-codimension consequence. A regular Riemannian foliation $\fol$ of codimension $q\in\{2,4\}$ on a compact connected manifold satisfies $\chi(\fol)>0$ whenever $\ric_{\fol}^{\intercal}>0$. Indeed, transverse Bonnet--Myers \cite[Theorem~2.6]{positivecaramellotoben} gives $\Ho^1(\fol)=0$. For $q=4$, apply this also to the transverse orientation cover. Basic Poincar\'e duality there \cite[Lemma~8.3]{positivecaramellotoben} gives vanishing in degree three; the finite-cover injection gives the same vanishing downstairs. Hence, in either codimension, $\chi(\fol)=\sum_j b_{2j}(\fol)\geq1$.

We also recall \cite[Lemma~8.8]{positivecaramellotoben}, which establishes that a closed, regular, transversely orientable Riemannian foliation $\fol$ of even codimension $q\geq2$ on a compact connected manifold $M$ satisfying $\sec_{\fol}>0$ and $\symrank(\fol)\geq q/4-1$ has $\chi(\fol)>0$. Here we observe that this also holds without transverse orientability. Indeed, its inductive proof uses this hypothesis only in the codimension-four base case, supplied here by the preceding low-codimension observation. To obtain positivity on the ambient manifold, when $q\geq6$ the symmetry-rank hypothesis guarantees a nonzero transverse Killing field. Its flow lifts to the transverse orientation cover, whose compact leaf orbifold is oriented and positively curved. The orbifold Synge--Weinstein theorem, applied to an element generating a dense subgroup of the closure of this
flow, gives a zero of the lifted field and hence of the original field. Euler localization \cite[Theorem~7.1]{positivecaramellotoben}, together with the positivity of the zero-set components established by the induction, then gives $\chi(\fol)>0$.

\begin{thm}\label{theorem positive curvature applications}
Suppose that $M$ is compact and
$\sec_{\fol}\geq\kappa>0$ on $M_{\reg}$.
\begin{enumerate}[(i)]
\item The defect satisfies
\[
d\leq\left\lfloor\frac{q+1}{2}\right\rfloor.
\]
\item If $q$ is even, then $\fol$ has a closed leaf.
If, moreover, $d\geq q/4-1$, then $\chi(\fol)>0$.
\item If $\actionarrow{\mathfrak{a}}{\fol}$ is equivariantly formal, there is
a constant $B(q)$, depending only on $q$, such that
\[
\sum_{j=0}^q b_j(\fol)\leq B(q).
\]
\item If equality holds in \textup{(i)}, then
\[
M/\overline{\fol}\cong\mathbb{D}^{q-d},
\]
where $\mathbb{D}^m$ denotes the closed $m$-dimensional ball. 
\end{enumerate}
\end{thm}

\begin{proof}
Corollary~\ref{corollary transverse curvature bounds resolution} gives $\sec_{\Gr\fol}\geq\kappa$. If $\Gr\fol$ is not transversely orientable, pass to the transversely oriented $\widehat{\Gr\fol}$ on a connected orientation double cover. As we saw in Section~\ref{section background} $\widehat{\Gr\fol}$ has the same structural algebra as $\Gr\fol$. Applying \cite[Theorem~A]{positivecaramellotoben} proves \textup{(i)}.

For the closed-leaf assertion in \textup{(ii)}, deform $\widehat{\Gr\fol}$ to a closed foliation with positive
transverse curvature. Its quotient is an even-dimensional, compact, oriented, positively curved orbifold with an isometric $\mathbb{T}^d$-action. If $d>0$, choose a torus element whose powers are dense in $\mathbb{T}^d$. The orbifold Synge--Weinstein theorem, in the form recalled in \cite[Section~5]{positivecaramellotoben}, gives a fixed point of this orientation-preserving isometry, hence a fixed point of the whole torus. Such points correspond to closed leaves of the initial Killing foliation, as in \cite[Section~5 and proof of Theorem~7.4]{positivecaramellotoben}. If $d=0$, that foliation is already closed. Projection of the finite cover gives a closed leaf on the resolution, and Lemma~\ref{lemma quotient comparison applications} then gives a closed leaf on $M$.

For the Euler assertion, take a sufficiently close closed approximation $\Gr\mathcal{G}$ of $\Gr\fol$ with positive transverse sectional curvature and symmetry rank at least $d$. By \cite[Lemma~8.8]{positivecaramellotoben}, with transverse orientability omitted as explained above, we obtain $\chi(\Gr\mathcal{G})>0$ when $q$ is even and $d\geq q/4-1$. Invariance of the basic Euler characteristic under this deformation \cite[Theorem~7.4]{positivecaramellotoben}, followed by Proposition~\ref{proposition basic forms resolution}, gives $\chi(\fol)>0$.

Proposition~\ref{proposition equivariant cohomology resolution} and \cite[Theorem~6.5]{equivariantcaramellotoben} prove \textup{(iii)}. For \textup{(iv)}, again choose a closed approximation of $\Gr\fol$ with positive transverse curvature, and denote its $q$-dimensional quotient orbifold by $\mathcal O$. It has an effective isometric $\mathbb{T}^d$-action and
\[M/\overline{\fol}\cong|\mathcal O|/\mathbb{T}^d.\]
Apply Harvey--Searle's maximal symmetry rank theorem \cite[Theorem~D]{harveysearle} to $|\mathcal O|$, rescaling to curvature at least $1$. The theorem includes Alexandrov spaces with boundary and identifies their torus actions equivariantly. Its two models are $\mathbb{S}^q/\Lambda$, with $\Lambda$ finite in the centralizer of a maximal torus of $\operatorname{O}(q+1)$, or, for even $q$, $\mathbb{S}^{q+1}/H$, with $H$ a rank-one subgroup of a maximal torus of $\operatorname{O}(q+2)$. The acting torus is induced by the standard maximal torus, up to a finite covering or kernel, which does not change its orbits.

If $q=2d-1$, the centralizer is $\mathbb{T}^d$ and the final torus quotient is
\[\mathbb{S}^{2d-1}/\mathbb{T}^d\cong\Delta^{d-1}\cong\mathbb{D}^{d-1}.\]
The first identification records the squared moduli of the $d$ complex coordinates, whose sum is $1$. If $q=2d$, write the sphere in $\mathbb{C}^d\oplus\mathbb{R}$. Its torus quotient is the suspension of $\Delta^{d-1}$, a closed $d$-ball. The centralizer is $\mathbb{T}^d\times\{\pm1\}$; after quotienting by the torus, $\Lambda$ acts either trivially or by exchanging the suspension tips. The latter quotient is a cone on $\Delta^{d-1}$, again a closed $d$-ball. In the rank-one model the final quotient is
\[\mathbb{S}^{2d+1}/\mathbb{T}^{d+1}\cong\Delta^d\cong\mathbb{D}^d.\]
These cases give \textup{(iv)}.
\end{proof}

In the even-dimensional rank-one model, a locally free circle action gives the usual finite quotient of a weighted projective space \cite[Corollary~E]{harveysearle}. Circle actions with fixed points also occur in the Alexandrov classification, including mirror orbifold models. The intrinsic conclusion in \textup{(iv)} is obtained after taking the whole torus quotient, which is a ball in all cases.

\begin{cor}\label{corollary maximal even defect applications}
If $M$ is compact and $q=2d$, then $\fol$ has only finitely many closed leaves and $\chi(\fol)$ is their number. Under the positive curvature hypothesis of Theorem~\ref{theorem positive curvature applications}, this number is positive.
\end{cor}

\begin{proof}
Apply \cite[Proposition~6.2]{positivecaramellotoben} to $\Gr\fol$ and use the correspondence of closed leaves in Lemma~\ref{lemma quotient comparison applications}. The quotient $M^{\mathfrak a}/\fol$ is then a finite set, and Corollary~\ref{corollary Euler localization applications} identifies its cardinality with $\chi(\fol)$. Under positive curvature, the even-codimension closed-leaf assertion gives positivity.
\end{proof}

The Ricci curvature comparisons give the following consequences.

\begin{cor}\label{corollary positive Ricci applications}
Let $\fol$ be a complete infinitesimally polar singular Riemannian foliation with $\ric_{\fol}^{\intercal}\geq\lambda$ on $M_{\reg}$, where $\lambda>0$. Then $M/\overline{\fol}$ is compact and $\Ho^1(\fol)=0$. If $M$ is compact and $q\in\{2,4\}$, then also $\chi(\fol)>0$. If $\fol$ is additionally Killing and $\chi(\fol)>0$, it has a closed leaf.
\end{cor}

\begin{proof}
The same Ricci bound holds on the complete resolution. The transverse Bonnet--Myers theorem \cite[Theorem~2.6]{positivecaramellotoben} gives compactness of its closure space and vanishing of its first basic cohomology. Transfer both conclusions by Lemma~\ref{lemma quotient comparison applications} and Proposition~\ref{proposition basic forms resolution}. If $M$ is compact and $q\in\{2,4\}$, apply the low-codimension Euler argument above to the compact regular foliation $\Gr\fol$. Preservation of basic cohomology gives $\chi(\fol)=\chi(\Gr\fol)>0$. If $\fol$ is Killing, Corollary~\ref{corollary Euler localization applications} supplies a closed leaf.
\end{proof}

We now generalize \cite[Theorem~4.3]{equivariantcaramellotoben} to the singular setting.

\begin{thm}[Nonpositive Ricci curvature and closedness]
\label{theorem negative curvature applications}
Suppose that $M$ is compact. If $\ric_{\fol}^{\intercal}\leq0$ on $M_{\reg}$ and $\ric_{\fol}^{\intercal}$ is negative definite at some point of $M_{\reg}$, then $\fol$ is closed. In particular, this holds under a uniform negative upper bound $\ric_{\fol}^{\intercal}\leq-\lambda$, with $\lambda>0$.
\end{thm}

\begin{proof}
The resolved transverse Ricci tensor is nonpositive everywhere and negative definite over the specified regular point. Pass, if necessary, to transverse and ambient orientation covers of the regular resolution. Recall the pullback $\widehat{\Gr\fol}$ is Killing with the same structural algebra $\mathfrak{a}$. Jung--Richardson's vanishing theorem \cite[Corollary~3.14]{jungrichardsongallotmeyer} applies after choosing a bundle-like metric whose basic mean curvature is coclosed. Their metric normalization changes only the leafwise metric and preserves the transverse metric \cite[Section~3, preceding Theorem~3.12]{jungrichardsongallotmeyer}; it therefore preserves both Ricci hypotheses. The theorem makes every global transverse conformal Killing field zero, in particular each global structural Killing field. The structural action is effective, so $\mathfrak{a}=0$, that is, $\widehat{\Gr\fol}$ is closed. Finite-cover projection and Lemma~\ref{lemma quotient comparison applications} give closedness of $\fol$.
\end{proof}

We now consider fundamental group growth. For a finitely generated group $G$ and a finite symmetric generating set $S$, let
\[B_S(m)=\{\gamma\in G:\ell_S(\gamma)\leq m\},\]
where $\ell_S(\gamma)$ is the minimum length of a word in $S$ representing $\gamma$. We say that $G$ has \emph{polynomial growth of degree at most $q$} if there is a constant $C_S>0$ such that
\[\#B_S(m)\leq C_S(1+m)^q,\]
for $m\in\mathbb N$. This property is independent of the finite generating set.

\begin{thm}[Nonnegative Ricci curvature and polynomial growth]
\label{theorem nonnegative Ricci growth applications}
Let $(M,\fol,\metric)$ be a complete, closed, infinitesimally polar singular Riemannian foliation of codimension $q\geq1$, with compact leaf space $M/\fol$. If $\ric_{\fol}^{\intercal}\geq0$ on $M_{\reg}$, then both
\[\pi_1^{\mathrm{orb}}(M\Slash\fol)\quad\text{and}\quad\pi_1(M/\fol)\]
have polynomial growth of degree at most $q$.
\end{thm}

\begin{proof}
By Lemma~\ref{lemma quotient comparison applications}, the resolved foliation is closed, and its leaf orbifold is isometric to $\mathcal O=M\Slash\fol$. This is a compact Riemannian orbifold of dimension $q$ by hypothesis. Corollary~\ref{corollary transverse curvature bounds resolution} gives $\ric_{\Gr\fol}^{\intercal}\geq0$ everywhere. Since the curvature of a regular foliation is that of its local quotients, the orbifold metric on $\mathcal O$ has nonnegative Ricci curvature, including at its singular points. Borzellino's theorem \cite[Proposition~10]{borzellinoRicci} therefore gives polynomial growth of degree at most $q$ for $\pi_1^{\mathrm{orb}}(\mathcal O)$.

The natural epimorphism
\[\pi_1^{\mathrm{orb}}(\mathcal O)\longrightarrow\pi_1(|\mathcal O|)=\pi_1(M/\fol)\]
forgets the orbifold information \cite[Remark~8]{borzellinoRicci}. If $S$ is a finite symmetric generating set of the first group, its image $T$ generates the second, and every word in $T$ lifts to a word of the same length in $S$. Consequently $\#B_T(m)\leq\#B_S(m)$, which proves the remaining assertion.
\end{proof}

\begin{thm}[Negative curvature and exponential growth]
\label{theorem negative sectional growth applications}
Suppose that $M$ is compact, $q\geq2$, and $-b^2\leq\sec_{\fol}\leq-a^2<0$ on $M_{\reg}$, for $0<a\leq b$. Then $\fol$ is closed and $\pi_1^{\mathrm{orb}}(M\Slash\fol)$ grows exponentially. Moreover, there is a constant $N(q)$, depending only on $q$, such that
\[\sum_{j=1}^q b_j(\fol)\leq N(q)b^q\vol(M\Slash\fol).\]
\end{thm}

\begin{proof}
The upper sectional curvature bound gives $\ric_{\fol}^{\intercal}\leq-(q-1)a^2$, so Theorem~\ref{theorem negative curvature applications} makes $\fol$ closed. Corollary~\ref{corollary transverse curvature bounds resolution} preserves the same sectional pinching on the compact regular resolution. Apply \cite[Theorem~4.6 and Corollary~4.7]{equivariantcaramellotoben} to obtain exponential growth of its orbifold fundamental group and the Betti number estimate. The resolved leaf orbifold is isometric to $M\Slash\fol$ by Lemma~\ref{lemma quotient comparison applications}, and the basic cohomology isomorphism of Proposition~\ref{proposition basic forms resolution} transfers the estimate.
\end{proof}

The deformation method also gives metric approximations and transverse sphere theorems.

\begin{thm}\label{theorem GH approximation applications}
Suppose that $M$ is compact. There is a sequence of compact $q$-dimensional Riemannian orbifolds $(\mathcal O_i,\mathrm{h}_i)$ such that
\[(|\mathcal O_i|,\mathrm d_i)\xrightarrow{\mathrm{GH}} M/\overline{\fol}\quad\text{and}\quad\chi(|\mathcal O_i|)=\chi(\fol),\]
where $\mathrm d_i$ is the induced quotient length metric. Any fixed strict constant bounds on transverse sectional or Ricci curvature that hold on $M_{\reg}$ with a uniform margin can be imposed on every $(\mathcal O_i,\mathrm{h}_i)$. Moreover, if $\actionarrow{\mathfrak{a}}{\fol}$ is equivariantly formal, then $b_j(|\mathcal O_i|)=b_j(\fol)$, for every $i$ and $j$.
\end{thm}

\begin{proof}
Apply \cite[Proposition~2.2, Lemma~3.3 and Theorem~5.2]{spherecaramelloneubauer} to $\Gr\fol$, and identify its metric leaf-closure space with $M/\overline{\fol}$. This gives closed regular foliations $\Gr\mathcal{G}_i$ and adapted metrics $\Gr\metric_i$ converging uniformly to $\Gr\metric$. Take $\mathcal O_i=\Gr M\Slash\Gr\mathcal{G}_i$, with the quotient metrics $\mathrm{h}_i$. The approximations arise from regular deformations, so their Euler characteristics, and under equivariant formality their Betti numbers, are those in
Theorem~\ref{theorem closed orbifold models applications}. Their transverse metrics and curvature tensors depend smoothly on the deformation parameters. Compactness therefore allows any fixed strict inequalities to be retained by choosing the parameters sufficiently close to the initial ones.
\end{proof}

\begin{thm}[Transverse diameter sphere theorem]
\label{theorem diameter sphere applications}
Suppose that $M$ is compact, $\sec_{\fol}\geq\kappa>0$ on $M_{\reg}$, and
\[\diam(M/\overline{\fol})>\frac{\pi}{2\sqrt\kappa}.\]
Then there are a finite subgroup $\Gamma<\operatorname{O}(q+1)$ and a continuous effective action $\actionarrow{\mathbb{T}^d}{\mathbb{S}^q/\Gamma}$ such that there is a homeomorphism
\[M/\overline{\fol}\cong(\mathbb{S}^q/\Gamma)/\mathbb{T}^d\]
In particular, if $\fol$ is closed, then $M/\fol$ is homeomorphic to $\mathbb{S}^q/\Gamma$.
\end{thm}

\begin{proof}
Put $X=M/\overline{\fol}$ and let $D=\diam X$. The resolution has curvature at least $\kappa$ and the same diameter $D$. Multiplying its metric by $\kappa$ gives transverse curvature at least $1$ and diameter $\sqrt\kappa D>\pi/2$. Choose
\[1<c<\left(\frac{2\sqrt\kappa D}{\pi}\right)^2.\]
Multiplying the resulting metric by $1/c$ gives curvature at least $c>1$ and diameter $\sqrt{\kappa/c}D>\pi/2$. We may therefore apply \cite[Theorem~4.1]{spherecaramelloneubauer} (whose curvature hypothesis is strict), and use Lemma~\ref{lemma quotient comparison applications}. Constant rescaling changes neither the leaves nor the structural algebra. The torus action is the isometric action on the quotient orbifold transported by the resulting homeomorphism to $\mathbb{S}^q/\Gamma$. If $\fol$ is closed, then $d=0$ and $X=M/\fol$.
\end{proof}

\begin{thm}[Transverse quarter-pinched sphere theorem]
\label{theorem quarter pinched applications}
Let $(M,\fol,\metric)$ be a connected, complete infinitesimally polar singular Riemannian foliation of codimension $q\geq3$. Suppose that there are constants $\alpha,\beta$ such that $0<\alpha\leq\sec_{\fol}\leq\beta<4\alpha$ on $M_{\reg}$. Then there are a Riemannian metric $\mathrm{h}$ on $\mathbb{S}^q$, satisfying the same
curvature bounds, and a countable subgroup $\Gamma<\operatorname{Iso}(\mathbb{S}^q,\mathrm{h})$ such that there is a homeomorphism $M/\fol\cong \mathbb{S}^q/\Gamma$ and an isometry $M/\overline{\fol}\cong(\mathbb{S}^q,\mathrm{h})/\overline\Gamma$. If $\fol$ is closed, then $\Gamma$ is finite and $M\Slash\fol$ is orbifold-diffeomorphic to $\mathbb{S}^q/\Gamma$.
\end{thm}

\begin{proof}
The complete regular foliation $\Gr\fol$ satisfies the same bounds. Choose $c$ with $\beta<c<4\alpha$ and multiply the resolved metric by $c$. Sectional curvature then lies in $[\alpha/c,\beta/c]\subset(1/4,1)$. Apply \cite[Theorem~5.3]{spherecaramelloneubauer} to this complete rescaled regular foliation. The resulting simple foliation is independent of the constant rescaling, and below we descend the original transverse metric, so the metric $\mathrm{h}$ has bounds $\alpha\leq\sec_{\mathrm{h}}\leq\beta$. Let $p\colon\widehat{\Gr M}\to\Gr M$ be the universal covering. The cited theorem makes $\widehat{\Gr\fol}$ simple, with leaves the connected fibers of a submersion
\[\varrho\colon\widehat{\Gr M}\longrightarrow \mathbb{S}^q.\]
Since the transverse metric is basic and the fibers are the leaves, it descends to a metric $\mathrm{h}$ on the sphere with the same curvature bounds. The pullback of the complete resolved metric is complete, and
$\varrho$ is a Riemannian submersion for this metric and $\mathrm{h}$.

Every element of $\mathrm{Deck}(p)$ preserves $\widehat{\Gr\fol}$ and its transverse metric. It thus induces a unique isometry of $(\mathbb{S}^q,\mathrm{h})$. Let $\Gamma$ be the image of this action. It is countable because it is an image of $\pi_1(\Gr M)$. Two fibers project to the same leaf downstairs exactly when their
base points are in the same $\Gamma$-orbit. The submersion $\varrho$ and the covering projection are open quotient maps, so the resulting bijection is a homeomorphism, even when the orbit space is non-Hausdorff. This identifies $\Gr M/\Gr\fol$ with $\mathbb{S}^q/\Gamma$, and Lemma~\ref{lemma quotient comparison applications} gives the first assertion on $M$.

Since $\widehat{\Gr M}$ is cmoplete and $\varrho$ is a Riemannian submersion, the distance between its fibers equals the distance between their base points. Taking the infimum over deck translates then identifies transverse distances with distances between sphere orbits. Since the closure of a $\Gamma$-orbit is its $\overline\Gamma$-orbit, the second assertion follows.

If $\fol$ is closed, all $\Gamma$-orbits are closed. Since $\Gamma$ is countable, each orbit of the compact Lie group $K=\overline\Gamma$ is countable and consequently has dimension zero. The identity component of $K$ fixes every point of the sphere and hence is trivial. Therefore $K$ is finite and $\Gamma=K$. The developing submersion then identifies the resolved transverse orbifold, and hence $M\Slash\fol$, with the finite global quotient $\mathbb{S}^q/\Gamma$.
\end{proof}

Finally, some hypotheses of the regular results we are transferring concern the topology of the ambient manifold. Their direct transfers thus necessarily involve $\Gr M$. We collect them in the following.

\begin{prop}
In items \textup{(i)--(iii)} below, the Killing assumption made at the beginning of this section is unnecessary: completeness and infinitesimal polarity suffice, in addition to the hypotheses stated in each item. In item \textup{(iv)} we retain the Killing assumption.
\begin{enumerate}[(i)]
\item If $q$ is even, $\sec_{\fol}\geq\kappa>0$ on $M_{\reg}$, and $\pi_1(\Gr M)$ is finite, then $\fol$ has a closed leaf.
\item If $M/\overline{\fol}$ is compact, $\ric_{\fol}^{\intercal}\leq-\lambda$ on $M_{\reg}$ for some $\lambda>0$, and $\pi_1(\Gr M)$ is finite, then $\fol$ is closed.
\item If $M$ is compact, $\pi_1(\Gr M)$ is finite, and the ordinary Euler characteristic $\chi(\Gr M)$ is nonzero, then $\fol$ is closed.
\item If $M$ is compact and $-b^2\leq\sec_{\fol}\leq-a^2<0$ on $M_{\reg}$ for constants $0<a\leq b$, then $\pi_1(\Gr M)$ grows exponentially.
\end{enumerate}
\end{prop}

\begin{proof}
For \textup{(i)--(iii)}, apply, respectively, \cite[Theorem~C]{positivecaramellotoben}, \cite[Corollary~4.4]{equivariantcaramellotoben}, and \cite[Theorem~F]{positivecaramellotoben} to the complete regular resolution. Compactness of the closure space in \textup{(ii)} transfers by Lemma~\ref{lemma quotient comparison applications}.
Compactness of the ambient manifold in \textup{(iii)} follows from properness. Closed leaves and closedness then descend by the same lemma.

For \textup{(iv)}, Theorem~\ref{theorem negative sectional growth applications} makes $\fol$ closed and gives exponential growth of $\pi_1^{\mathrm{orb}}(M\Slash\fol)$. The regular resolution has the same leaf orbifold and a surjection
\[\pi_1(\Gr M)\longrightarrow\pi_1^{\mathrm{orb}}(M\Slash\fol)\]
(see \cite[Section~2.1, equation~(2.1), and Theorem~4.6]{equivariantcaramellotoben}). A finitely generated group surjecting onto an exponentially growing group also grows exponentially, proving the assertion.
\end{proof}

The distinction between $M$ and $\Gr M$ is necessary in using this argument. For example, the orbit foliation of the standard circle action on $\mathbb{S}^2$ is infinitesimally polar and singular Killing. Its geometric resolution is the real projective blow-up at the two poles, which is a Klein bottle: a neighborhood of each exceptional circle is a Möbius band. Hence $\pi_1(\mathbb{S}^2)$ is trivial and $\chi(\mathbb{S}^2)=2$, whereas the resolution has infinite fundamental group and Euler characteristic zero. Thus finiteness of $\pi_1(M)$ and $\chi(M)\neq0$ do not by themselves supply the corresponding hypotheses upstairs.

\section*{Acknowledgements}
\noindent L.~R.~S. acknowledges support from the Fundação de Amparo à Pesquisa
do Estado de São Paulo (FAPESP), grant 2025/27715-7.


\begin{thebibliography}{99}

\bibitem{alexdesing}
Alexandrino, M.~M.
\emph{Desingularization of singular Riemannian foliation},
Geom. Dedicata \textbf{149} (2010), 397--416.
\url{https://doi.org/10.1007/s10711-010-9489-4}.

\bibitem{surveyalexcaramello}
Alexandrino, M.~M. and Caramello Jr., F.~C.
\emph{Leaf closures of Riemannian foliations: a survey on topological
and geometric aspects of Killing foliations},
Expo. Math. \textbf{40} (2022), 177--230.
\url{https://doi.org/10.1016/j.exmath.2021.11.002}.

\bibitem{alexlytchak}
Alexandrino, M.~M. and Lytchak, A.
\emph{On smoothness of isometries between orbit spaces},
in: Riemannian Geometry and Applications---Proceedings RIGA (2011),
Ed. Univ. Bucure\c{s}ti, Bucharest, 2011, 17--28.
\href{https://arxiv.org/abs/1107.2563}{arXiv:1107.2563}.

\bibitem{alexradeschiclosure}
Alexandrino, M.~M. and Radeschi, M.
\emph{Closure of singular foliations: the proof of Molino's conjecture},
Compos. Math. \textbf{153} (2017), 2577--2590.
\url{https://doi.org/10.1112/S0010437X17007485}.

\bibitem{alexradeschiflows}
Alexandrino, M.~M. and Radeschi, M.
\emph{Smoothness of isometric flows on orbit spaces and applications},
Transform. Groups \textbf{22} (2017), 1--27.
\url{https://doi.org/10.1007/s00031-016-9386-5}.

\bibitem{borzellinoRicci}
Borzellino, J.~E.
\emph{Orbifolds with lower Ricci curvature bounds},
Proc. Amer. Math. Soc. \textbf{125}(10) (1997), 3011--3018.
\url{https://doi.org/10.1090/S0002-9939-97-04046-X}.

\bibitem{boualem}
Boualem, H.
\emph{Feuilletages riemanniens singuliers transversalement int\'egrables},
Compos. Math. \textbf{95} (1995), 101--125.
\href{https://www.numdam.org/item/CM_1995__95_1_101_0/}{Numdam}.

\bibitem{caramellosingularequivariant}
Caramello Jr., F.~C.
\emph{Equivariant basic cohomology of singular Riemannian foliations},
Monatsh. Math. \textbf{203} (2024), 779--808.
\url{https://doi.org/10.1007/s00605-023-01902-6}.

\bibitem{spherecaramelloneubauer}
Caramello Jr., F.~C. and Neubauer, F.~A.
\emph{Transverse sphere theorems for Riemannian foliations},
Preprint,
\href{https://arxiv.org/abs/2505.01378v2}{arXiv:2505.01378v2}
(2026).


\bibitem{blupcaramello}
Caramello Jr., F.~C. and Ribeiro dos Santos, L.
\emph{Blown-up singular Riemannian foliations},
Preprint,
\href{https://arxiv.org/abs/2512.07069v4}{arXiv:2512.07069v4}
(2026).


\bibitem{equivariantcaramellotoben}
Caramello Jr., F.~C. and T\"oben, D.
\emph{Equivariant basic cohomology under deformations},
Math. Z. \textbf{299} (2021), 2461--2482.
\url{https://doi.org/10.1007/s00209-021-02768-w}.

\bibitem{positivecaramellotoben}
Caramello Jr., F.~C. and T\"oben, D.
\emph{Positively curved Killing foliations via deformations},
Trans. Amer. Math. Soc. \textbf{372} (2019), 8131--8158.
\url{https://doi.org/10.1090/tran/7893}.

\bibitem{haefligersalem}
Haefliger, A. and Salem, E.
\emph{Riemannian foliations on simply connected manifolds
and actions of tori on orbifolds},
Illinois J. Math. \textbf{34} (1990), 706--730.
\url{https://doi.org/10.1215/ijm/1255988064}.

\bibitem{harveysearle}
Harvey, J. and Searle, C.
\emph{Orientation and symmetries of Alexandrov spaces
with applications in positive curvature},
J. Geom. Anal. \textbf{27} (2017), 1636--1666.
\url{https://doi.org/10.1007/s12220-016-9734-7}.

\bibitem{jungrichardsongallotmeyer}
Jung, S.~D. and Richardson, K.
\emph{Transverse conformal Killing forms
and a Gallot--Meyer theorem for foliations},
Math. Z. \textbf{270} (2012), 337--350.
\url{https://doi.org/10.1007/s00209-010-0800-8}.

\bibitem{lytchak}
Lytchak, A.
\emph{Geometric resolution of singular Riemannian foliations},
Geom. Dedicata \textbf{149} (2010), 379--395.
\url{https://doi.org/10.1007/s10711-010-9488-5}.

\bibitem{lytchakthorbergsson}
Lytchak, A. and Thorbergsson, G.
\emph{Curvature explosion in quotients and applications},
J. Differential Geom. \textbf{85} (2010), 117--140.
\url{https://doi.org/10.4310/jdg/1284557927}.

\bibitem{moerdijkmrcun}
Moerdijk, I. and Mr\v{c}un, J.
\emph{Introduction to Foliations and Lie Groupoids},
Cambridge Studies in Advanced Mathematics \textbf{91},
Cambridge University Press, Cambridge, 2003.


\bibitem{molinodesing}
Molino, P.
\emph{D\'esingularisation des feuilletages riemanniens},
Amer. J. Math. \textbf{106} (1984), 1091--1106.
\url{https://doi.org/10.2307/2374274}.

\bibitem{molino}
Molino, P.
\emph{Riemannian Foliations},
Progress in Mathematics \textbf{73},
Birkh\"auser, Boston, 1988.


\bibitem{petersen}
Petersen, P.
\emph{Riemannian Geometry},
3rd ed., Graduate Texts in Mathematics \textbf{171},
Springer, Cham, 2016.


\bibitem{schwarzlifting}
Schwarz, G.~W.
\emph{Lifting smooth homotopies of orbit spaces},
Publ. Math. IH\'ES \textbf{51} (1980), 37--135.
\url{https://doi.org/10.1007/BF02684776}.

\bibitem{tobenparallel}
T\"oben, D.
\emph{Parallel focal structure and singular Riemannian foliations},
Trans. Amer. Math. Soc. \textbf{358} (2006), 1677--1704.
\url{https://doi.org/10.1090/S0002-9947-05-03829-8}.

\end{thebibliography}
\end{document}